\documentclass[11pt,a4paper]{amsart}  

\usepackage{comment}
\usepackage{tikz-cd}
\usepackage{amsmath,amsthm,amssymb,enumerate,amsfonts,graphicx}

\usepackage{subcaption}
\usepackage{algorithm,algorithmic}

\usepackage{color}
\usepackage{epsfig}
\usepackage[numbers,sort&compress]{natbib}
\usepackage[T1]{fontenc}
\usepackage{algorithm,algorithmic}
\usepackage{stmaryrd}
\usepackage{paralist}
\usepackage{geometry}
\usepackage{tikz}
\usepackage{todonotes}
\tikzset{black node/.style={draw, circle, fill = black, minimum size = 5pt, inner sep = 0pt}}
\tikzset{normal/.style = {draw=none, fill = none, minimum size =0, rectangle}}

\makeatletter
\newtheorem*{rep@theorem}{\rep@title}
\newcommand{\newreptheorem}[2]{%
\newenvironment{rep#1}[1]{%
 \def\rep@title{#2 \ref{##1}}%
 \begin{rep@theorem}}%
 {\end{rep@theorem}}}
\makeatother

\newreptheorem{theorem}{Theorem}

\usepackage{aliascnt}

\newtheorem{theorem}{Theorem}

\newaliascnt{lemma}{theorem}

\aliascntresetthe{lemma}

\newaliascnt{observation}{theorem}

\aliascntresetthe{observation}

\newaliascnt{corollary}{theorem}

\aliascntresetthe{corollary}

\newaliascnt{conjecture}{theorem}

\aliascntresetthe{conjecture}

\newaliascnt{claim}{theorem}

\aliascntresetthe{claim}

\theoremstyle{definition}

\newcommand{\N}{\mathbb{N}}

\DeclareMathOperator{\OPT}{\mathrm{OPT}}

\newcommand{\svd}{\textsc{SVD}}

\begin{document}

\title[A simple $(2+\epsilon)$-approximation algorithm for Split Vertex Deletion]{A simple $(2+\epsilon)$-approximation algorithm for Split Vertex Deletion}

\author[M.~Drescher]{Matthew Drescher}
\author[S.~Fiorini]{Samuel Fiorini}
\author[T.~Huynh]{Tony Huynh}
\address[M.~Drescher, S.~Fiorini]{\newline D\'epartement de Math\'ematique
\newline Universit\'e libre de Bruxelles
\newline Brussels, Belgium}
\email{knavely@gmail.com, sfiorini@ulb.ac.be}
\address[T.~Huynh]{\newline School of Mathematics
\newline Monash University
\newline Melbourne, Australia}
\email{tony.bourbaki@gmail.com}

\thanks{Tony Huynh is supported by the Australian Research Council.}

\date{\today}
\sloppy

\begin{abstract}
A  \emph{split graph} is a graph whose vertex set can be partitioned into a clique and a stable set.  Given a graph $G$ and weight function $w: V(G) \to \mathbb{Q}_{\geq 0}$, the \textsc{Split Vertex Deletion} (\svd) problem asks to find a minimum weight set of vertices $X$ such that $G-X$ is a split graph. It is easy to show that a graph is a split graph if and only it it does not contain a $4$-cycle, $5$-cycle, or a two edge matching as an induced subgraph.  Therefore, \svd{} admits an easy $5$-approximation algorithm.  On the other hand, for every $\delta >0$, \svd{} does not admit a $(2-\delta)$-approximation algorithm, unless P=NP or the Unique Games Conjecture fails.  

For every $\epsilon >0$, Lokshtanov, Misra, Panolan, Philip, and Saurabh~\cite{lokshtanov20202+} recently gave a \emph{randomized} $(2+\epsilon)$-approximation algorithm for \svd{}.  In this work we give an extremely simple deterministic $(2+\epsilon)$-approximation algorithm for \svd{}.  
\end{abstract}

\maketitle

A graph $G$ is a \emph{split graph} if $V(G)$ can be partitioned into two sets $K$ and $S$ such that $K$ is a clique and $S$ is a stable set. Split graphs are an important subclass of perfect graphs which feature prominently in the proof of the Strong Perfect Graph Theorem by Chudnovsky, Robertson, Seymour, and Thomas~\cite{CRST06}.  

 Given a graph $G$ and weight function $w: V(G) \to \mathbb{Q}_{\geq 0}$, the \textsc{Split Vertex Deletion} (\svd) problem asks to find a set of vertices $X$ such that $G-X$ is a split graph and $w(X):=\sum_{x \in X} w(x)$ is minimum.   A subset $X \subseteq V(G)$ such that $G-X$ is a split graph is called a \emph{hitting set}. We denote by $\OPT(G,w)$ the minimum weight of a hitting set. 
 
 It is easy to show $G$ is a split graph if and only if $G$ does not contain $C_4, C_5$ or $2K_2$ as an induced subgraph, where $C_\ell$ denotes a cycle of length $\ell$ and $2K_2$ is a matching with two edges. Therefore, the following is an easy $5$-approximation algorithm\footnote{An \emph{$\alpha$-approximation algorithm} for \svd{} is a (deterministic) polynomial-time algorithm computing a hitting set $X$ with $w(X) \leqslant \alpha \cdot \OPT(G,w)$.} for \svd{} in the unweighted case (the general case follows from the \emph{local ratio method}~\cite{bbfr2004}). If $G$ is a split graph, then $\varnothing$ is a hitting set, and we are done.  Otherwise, we find an induced subgraph $H$ of $G$ such that $H \in \{C_4, C_5, 2K_2\}$. We put $V(H)$ into the hitting set, replace $G$ by $G-V(H)$, and recurse. 
 
 On the other hand, there is a simple approximation preserving reduction from \textsc{Vertex Cover} to \svd{} (see~\cite{lokshtanov20202+}).  Therefore, for every $\delta >0$, \svd{} does not admit a $(2-\delta)$-approximation algorithm, unless P=NP or the Unique Games Conjecture fails~\cite{KR08}.  
 
 For every $\epsilon >0$, Lokshtanov, Misra, Panolan, Philip, and Saurabh~\cite{lokshtanov20202+} recently gave a \emph{randomized} $(2+\epsilon)$-approximation algorithm for \svd{}.  Their approach is based on the randomized $2$-approximation algorithm for feedback vertex set in tournaments~\cite{LMMPPS20}, but is more complicated and requires several new ideas and insights. 
 
 Here we give a much simpler deterministic $(2+\epsilon)$-approximation algorithm for \svd{}.  
 
 \begin{theorem} \label{main}
 For every $\epsilon >0$, there is a (deterministic) $(2+\epsilon)$-approximation algorithm for \svd{}.
 \end{theorem}
 
 As far as we can tell, the easy $5$-approximation described above was the previously best (deterministic) approximation algorithm for \svd{}.
 Before describing our algorithm and proving its correctness, we need a few definitions. 
 
 Let $G$ be a graph and $\mathcal H$ be a family of graphs.  We say that $G$ is \emph{$\mathcal H$-free} if $G$ does not contain $H$ as an induced subgraph for all $H \in \mathcal H$. We let $\overline{G}$ be the complement of $G$. A \emph{cut} in a graph $G$ is a pair $(A,B)$ such that $A \cup B = V(G)$ and $A \cap B = \varnothing$. The cut $(A,B)$ is said to \emph{separate} a pair $(K,S)$ where $K$ is a clique, and $S$ a stable set if $K \subseteq A$ and $S \subseteq B$. A family of cuts $\mathcal{F}$ is called a \emph{clique-stable set separator} if for all pairs $(K, S)$ where $K$ is a clique and $S$ is a stable set disjoint from $K$, there exists a cut $(A,B)$ in $\mathcal{F}$ such that $(A,B)$ separates $(K,S)$.  For each $k \in \N$, let $P_k$ be the path on $k$ vertices. 
 
 The main technical ingredient we require is the following theorem of Bousquet, Lagoutte and Thomass\'e \cite{BOUSQUET201473}
 
 \begin{theorem} \label{separator}
 For every $k \in \N$, there exists $c(k) \in \N$ such that every $n$-vertex, $\{P_k,\overline{P_k}\}$-free graph has a clique-stable set separator of size at most $n^{c(k)}$.  Moreover, such a clique-stable set separator can be found in polynomial time.\footnote{We remark that~\cite{BOUSQUET201473} do not state that the clique-stable set separator can be found in polynomial time, but this is easy to check, where the relevant lemmas appear in \cite[Theorem 4]{BLA15}, ~\cite[Theorem 1.1]{FS08}, and~\cite[Lemma 1.5]{EH89}.  Note that the abstract of ~\cite{BOUSQUET201473} states that $c(k)$ is a tower function. However, the bound for $c(k)$ can be significantly improved by using ~\cite[Theorem 1.1]{FS08} instead of a lemma of R\"{o}dl~\cite{Rodl86} (which was used in an older version of~\cite{BLA15}).  The proof of~\cite[Theorem 1.1]{FS08} does not use the Szemerédi Regularity Lemma~\cite{szemeredi78}, and provides much better quantitative estimates.}  
 \end{theorem}

 We are now ready to state and prove the correctness of our algorithm.
 
 \begin{proof}[Proof of Theorem~\ref{main}]
 Let $G$ be an $n$-vertex graph, $w: V(G) \to \mathbb{Q}_{\geq 0}$, and $\epsilon >0$.  We may assume that $w(v)>0$ for all $v \in V(G)$, since we may delete vertices of weight $0$ for free.  Choose $k$ sufficiently large so that $\frac{2k}{k-4} \leq 2+\epsilon$. Let $\mathbf{1}$ be the weight function on $V(P_k)$ which is identically $1$.   Since the largest clique of $P_k$ has size $2$ and every vertex cover of $P_k$ has size at least $\lfloor k / 2 \rfloor$, every hitting set of $P_k$ has size at least $\frac{k-4}{2}$.  Therefore, $|V(P_k)| / \OPT(P_k, \mathbf{1}) \leq 2+\epsilon$, and so by the local ratio method~\cite{bbfr2004}, we may assume that $G$ is $P_k$-free.  Note that $G$ is a split graph if and only if $\overline{G}$ is a split graph.  Thus, we may also assume that $G$ is $\overline{P_k}$-free.  Now, by Theorem~\ref{separator}, there exists a constant $c(k)$ such that $G$ has a clique-stable set separator $\mathcal F$ such that $|\mathcal F| \leq n^{c(k)}$.  
 
 For each $(A,B) \in \mathcal F$, let $\rho_A$ and $\rho_B$ be the weights of the minimum vertex covers of $(\overline{G}[A], w)$ and $(G[B], w)$.  Since there is a $2$-approximation algorithm for vertex cover, for each $(A,B) \in \mathcal F$, we can find vertex covers $X_A$ and $X_B$ of $(\overline{G}[A], w)$ and $(G[B], w)$ such that $w(X_A) \leq 2\rho_A$ and $w(X_B) \leq 2\rho_B$.  Let $X^*$ be a minimum weight hitting set for $(G, w)$, and suppose that $V(G-X^*)$ is partitioned into a clique $K^*$ and a stable set $S^*$.  Since $\mathcal F$ is a clique-stable set separator, there must be some $(A^*,B^*) \in \mathcal F$ such that $K^* \subseteq A^*$ and $S^* \subseteq B^*$.  Therefore, if we choose $(A,B) \in \mathcal F$ such that $w(X_A)+w(X_B)$ is minimum, then $X_A \cup X_B$ is a hitting set such that $w(X_A \cup X_B) \leq 2w(X^*)$.  Finally, since $|\mathcal F| \leq n^{c(k)}$, our algorithm clearly runs in polynomial time.   
 \end{proof}
 
\bibliographystyle{abbrv}
\bibliography{references}

\begin{thebibliography}{10}

\bibitem{BOUSQUET201473}
N.~Bousquet, A.~Lagoutte, and S.~Thomassé.
\newblock Clique versus independent set.
\newblock {\em European Journal of Combinatorics}, 40:73 -- 92, 2014.

\bibitem{BLA15}
N.~Bousquet, A.~Lagoutte, and S.~Thomass\'{e}.
\newblock The {E}rdős-{H}ajnal conjecture for paths and antipaths.
\newblock {\em J. Combin. Theory Ser. B}, 113:261--264, 2015.

\bibitem{CRST06}
M.~Chudnovsky, N.~Robertson, P.~Seymour, and R.~Thomas.
\newblock The strong perfect graph theorem.
\newblock {\em Ann. of Math. (2)}, 164(1):51--229, 2006.

\bibitem{EH89}
P.~Erdős and A.~Hajnal.
\newblock Ramsey-type theorems.
\newblock volume~25, pages 37--52. 1989.
\newblock Combinatorics and complexity (Chicago, IL, 1987).

\bibitem{FS08}
J.~Fox and B.~Sudakov.
\newblock Induced {R}amsey-type theorems.
\newblock {\em Adv. Math.}, 219(6):1771--1800, 2008.

\bibitem{bbfr2004}
A.~Freund, R.~Bar-Yehuda, and K.~Bendel.
\newblock Local ratio: a unified framework for approximation algorithms.
\newblock {\em ACM Computing Surveys}, 36:422--463, 01 2005.

\bibitem{KR08}
S.~Khot and O.~Regev.
\newblock Vertex cover might be hard to approximate to within {$2-\epsilon$}.
\newblock {\em J. Comput. System Sci.}, 74(3):335--349, 2008.

\bibitem{LMMPPS20}
D.~Lokshtanov, P.~Misra, J.~Mukherjee, F.~Panolan, G.~Philip, and S.~Saurabh.
\newblock $2$-approximating feedback vertex set in tournaments.
\newblock In {\em Proceedings of the Fourteenth Annual ACM-SIAM Symposium on
  Discrete Algorithms}, pages 1010--1018. SIAM, 2020.

\bibitem{lokshtanov20202+}
D.~Lokshtanov, P.~Misra, F.~Panolan, G.~Philip, and S.~Saurabh.
\newblock A (2+ $\varepsilon$)-factor approximation algorithm for split vertex
  deletion.
\newblock In {\em 47th International Colloquium on Automata, Languages, and
  Programming (ICALP 2020)}. Schloss Dagstuhl-Leibniz-Zentrum f{\"u}r
  Informatik, 2020.

\bibitem{Rodl86}
V.~Rödl.
\newblock On universality of graphs with uniformly distributed edges.
\newblock {\em Discrete Math.}, 59(1-2):125--134, 1986.

\bibitem{szemeredi78}
E.~Szemer\'{e}di.
\newblock Regular partitions of graphs.
\newblock In {\em Probl\`emes combinatoires et th\'{e}orie des graphes
  ({C}olloq. {I}nternat. {CNRS}, {U}niv. {O}rsay, {O}rsay, 1976)}, volume 260
  of {\em Colloq. Internat. CNRS}, pages 399--401. CNRS, Paris, 1978.

\end{thebibliography}

\end{document}